%% file: multiplicity-latex.tex
\documentclass[
final
, nomarks
]{dmtcs-episciences}

\input{mathmacros.tex}

\usepackage[utf8]{inputenc}
\usepackage{subfigure}
\usepackage{amsthm}
\usepackage{amsmath}
\allowdisplaybreaks
\makeatletter
\let\over\@@over
\let\atop\@@atop
\let\atopwithdelims\@@atopwithdelims
\makeatother

\newtheorem{thm}{Theorem}
\newtheorem{lem}[thm]{Lemma}
\let\[\relax \let\]\relax
\DeclareRobustCommand{\[}{\begin{equation}}
\DeclareRobustCommand{\]}{\end{equation}}

\author{Anna M. Brandenberger
  \and Luc Devroye
  \and Marcel K. Goh
  \and Rosie Y. Zhao}
\title{Leaf multiplicity in a Bienaym\'e--Galton--Watson tree}

\affiliation{
School of Computer Science, McGill University, Montr\'eal}
\received{2021-05-26}
\revised{2022-02-18}
\accepted{2022-02-19}
\begin{document}
\publicationdetails{24}{2022}{1}{7}{7515}
\maketitle
\begin{abstract}
This note defines a notion of multiplicity for nodes in a rooted tree and presents an
asymptotic calculation of the maximum multiplicity over all
leaves in a Bienaym\'e--Galton--Watson tree with critical offspring distribution $\xi$,
conditioned on the tree being of size $n$. In particular, we show that if $S_n$ is the
maximum multiplicity in a conditioned Bienaym\'e--Galton--Watson tree, then
$S_n = \Omega(\log n)$ asymptotically in probability and under the further assumption
that $\ex\{2^\xi\} < \infty$, we have $S_n = O(\log n)$ asymptotically in probability as well.
Explicit formulas are given for the constants in both bounds. We conclude by discussing links
with an alternate definition of multiplicity that arises in the root-estimation problem.
\end{abstract}

\section{Introduction}
\label{sec:intro}

Equivalence between two distinct mathematical objects is a far-reaching concept in mathematics. When two structures are similar, one may define a relation under which they are regarded as one and the same. The term ``multiplicity'' is often used to indicate the extent to which an object is, in some sense, not structurally unique (or how often it is repeated in a suitably-defined multiset).
Towards a concept of the multiplicity of a node in a tree, consider the small example depicted in
Fig.~\ref{fig:multdef}

\begin{figure}
\begin{center}
  \subfigure{\includegraphics{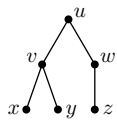}}
  \caption{A rooted tree with six nodes. The pair $x$ and $y$ are similar, but $x$ and $z$ are not.}
  \label{fig:multdef}
\end{center}
\end{figure}

\subsection{Definitions and notation}
Consider a tree $T$ rooted at
a node $u$. For a node $v$ in the tree, we let $T_v$ denote the subtree rooted at $v$. Let $v$ and $w$ be nodes in the tree and let $v=v_1, v_2,\ldots, v_n=u$ and $w=w_1, w_2, \ldots, w_m=u$ be the paths from $v$ and $w$, respectively, to the root. We say that $v$ and $w$ are {\it identical} and write $v\equiv w$ if the paths have the same length and $T_{v_j}$ and $T_{w_j}$ are isomorphic as rooted ordered trees for $1\leq j\leq n$.
For example, in Fig.~\ref{fig:multdef}, the nodes $x$ and $y$ are identical in this sense, but different from node $z$.

It is clear that $\equiv$ defines an equivalence relation on the set of nodes in the tree, so we may now define the {\it multiplicity} $\sigma(v)$ of a node $v$ to be the size of the equivalence class $[v]$ under the relation. The
{\it leaf multiplicity} (or simply {\it multiplicity}, when no confusion can arise)
$S(T)$ of a rooted tree $T$ is the maximum value of $\sigma(v)$, taken over all nodes $v$ of $T$. The name ``leaf multiplicity'' is motivated by the fact that the function $\sigma$ increases monotonically away from the root, so that $S(T)$ remains the same when the maximum is only computed over the set of {\it leaves} of $T$.

Note that $\equiv$ is not the only structural equivalence relation one can define on the set of nodes in a
tree, and thus $\sigma$ is only one of many possible notions of leaf multiplicity.
Towards the end of this paper, we will explore an alternate definition $\mu$ of multiplicity which extends well
to free trees as well as rooted trees, and discuss the relationship between the two functions $\sigma$ and $\mu$.

\subsection{The Bienaym\'e--Galton--Watson model}

For a random variable $\xi$ taking values in the nonnegative integers, a {\it Bienaym\'e--Galton--Watson tree}~\cite{athreya1972branching} is a rooted ordered tree in which every node has $i$ children with probability $p_i = \pr\{\xi = i\}$. We say that $\xi$ is the {\it offspring distribution} of the tree.

Trees arising from this process are often called {\it Galton--Watson trees}, but we include the name of I.-J.~Bienaym\'e~\cite{bienayme1845}, since his work predates (and is more mathematically precise than) the analysis of F. Galton and H. W. Watson~\cite{galtonwatson1874}; see~\cite{jagers2009} for an extended account of the history of branching processes. We shall deal with critical Bienaym\'e--Galton--Watson trees; that is, trees whose offspring distributions satisfy $\ex\{\xi\} = 1$ and $\var\{\xi\} \in (0,\infty)$. This restriction on the variance ensures that $p_i\neq 1$, so that the tree is finite almost surely. The Bienaym\'e--Galton--Watson trees that we shall study are {\it conditioned Bienaym\'e--Galton--Watson trees} $T_n$.  Such trees are conditioned on $|T| = n$, where $|T|$ is the number of nodes in the tree.

\subsection{R\'enyi entropy}
It will be convenient to simplify our notation with some information-theoretic
definitions. Letting $p_i = \pr\{\xi = i\}$, for $\alpha>1$
we define the {\it R\'enyi entropy of order $\alpha$}~\cite{renyi1961measures} (see also~\cite{jacquet1999entropy}) to be the value
\[\Eta_\alpha(\xi) = {1\over \alpha - 1}\log_2 {1\over \sum_{i\geq 0} {p_i}^\alpha}.\]
As $\alpha\to 1$, this approaches the {\it binary (Shannon) entropy}~\cite{shannon1948}
\[\Eta(\xi) = \sum_{i\geq 0} p_i \log_2 {1\over p_i}.\]
Since $\xi$ will be fixed throughout the paper, for brevity we will let $\Eta_\alpha = \Eta_\alpha(\xi)$ and
$\Eta = \Eta(\xi)$.

Fix an offspring distribution $\xi$ with mean $1$ and nonzero finite variance;
let $T_n$ be a conditioned Bienaym\'e--Galton--Watson tree of size $n$ with this offspring distribution. The leaf
multiplicity $S(T_n)$ of this tree is a random variable, and will be denoted $S_n$.
The main result of this paper gives bounds on $S_n$ that are obeyed asymptotically in probability.

\begin{thm}
\label{mainthm}
Let $\xi$ be an offspring distribution with $\ex\{\xi\}=1$
and $\var\{\xi\}\in (0,\infty)$. If $S_n$ is the multiplicity of a conditioned Bienaym\'e--Galton--Watson tree of size $n$ with offspring distribution $\xi$, then letting
\[\gamma = \max_{k \geq 2}{p_0}^k {p_k}^{k/(k-1)},\]
we have for all $\eps > 0$, 
\[\pr \bigg\{ S_n \geq (1-\eps) {\log_2 n \over \log_2(1/\gamma)} \bigg\} \to 1\]
as $n\to\infty$, and under the further assumption
that $\ex\{2^\xi\}<\infty$, we have the upper bound
\[\pr \bigg\{ S_n \leq (1+\eps) {2\log_2 n\over \Eta_2} \bigg\} \to 1\]
as $n\to\infty$, where $\Eta_2$ is the R\'enyi entropy of order $2$
of the random variable $\xi$.
\end{thm}

This theorem will be proved as two separate lemmas in the next section.

\section{Asymptotics of the leaf multiplicity}
\label{sec:asymptotics}

In this section we derive asymptotic
upper and lower bounds on $S_n$. Before we begin, we observe that if
$\pmax = \max_{i\geq 0} p_i$ and $1<\alpha < \beta < \infty$, we have the inequalities
\[e^{-\Eta} \leq \Big(\sum_{i\geq 0} {p_i}^\alpha\Big)^{1/(\alpha-1)} \leq \Big(\sum_{i\geq 0} {p_i}^\beta\Big)^{1/(\beta-1)} \leq \pmax \leq 
\Big(\sum_{i\geq 0} {p_i}^\beta\Big)^{1/\beta} \leq \Big(\sum_{i\geq 0} {p_i}^\alpha\Big)^{1/\alpha} \leq 1,\]
and defining $\Eta_\infty = \log_2(1/\pmax)$, we have the equivalent chain of inequalities
\[\Eta \geq \Eta_\alpha \geq \Eta_\beta \geq \Eta_\infty \geq {\beta-1 \over \beta} \Eta_\beta \geq {\alpha-1\over\alpha}\Eta_\alpha\geq 0.\]
In particular, because we have assumed that $\var\{\xi\}\neq 0$, we have the strict inequality
\[\Big(\sum_{i\geq 0} {p_i}^k\Big)^{1/ k} < \Big(\sum_{i\geq 0} {p_i}^2\Big)^{1/2}\label{strictineq}\]
for all $k> 2$. 

We will prove the upper bound and lower bound for $S_n$ as two separate lemmas.

\begin{lem}
\label{upperbound}
Let $\xi$ be an offspring distribution with mean 1 and nonzero finite variance $\sigma^2$. Suppose further that $\ex\{2^\xi\}$ is finite.
If $S_n$ is the multiplicity of a conditioned Bienaym\'e--Galton--Watson tree of size $n$ with offspring
distribution $\xi$, then
\[\pr\{S_n>(1+\eps)2\log_2 n/\Eta_2\} \to 0 \]
for all $\eps>0$, where $\Eta_2$ is the R\'enyi entropy of order $2$ of the random variable $\xi$.
\end{lem}

\begin{proof}
For $1\leq i\leq n$, let $\xi_i$ denote the degree of the $i$th node in preorder in the tree $T_n$.
For all $1\leq t< n$, the partial sum $\sum_{i=1}^t \xi_i > t-1$ and $\sum_{i=1}^n \xi_i = n-1$.
We will concentrate on the least common ancestor of the nodes in the largest equivalence
class of $T_n$. This node, call it $w$, has the property that the nodes in the equivalence class
belong to $k\geq 2$ different subtrees rooted at the children of $w$. The node $w$ has (random) degree $D$,
which we will deal with by summing over all possible degrees $d$.
Let $\A_{wk}$ denote the collection of all
subsets of size $k$ of the children of $w$ (naturally, this collection is empty if $w$ has fewer than $k$
children). For $x>0$, a node $w$,
integers $2\leq k\leq d$,
and a set $A\in \A_{wk}$, we let $E(x,w,k,A)$ be the event that all the nodes in $A$
are identical and their subtree sizes are at least $x/k$. Now for integers $s\geq x/k$,
we let $E'(x,k,d,s,A)$ be the event that a randomly chosen node $w$ of the tree $T_n$ has degree $d$ and
the leftmost $k$ children of $w$ are identical, with subtrees of size $s$. We have, by the union bound,
\[\pr\{S_n\geq x\} \leq \pr\!\bigg\{\bigcup_{w\in T_n} \bigcup_{k\geq 2}
\bigcup_{A\in \A_{wk}} E(x,w,k,A)\bigg\}\leq n\sum_{k\geq 2} \sum_{d\geq k} {d\choose k}\sum_{s\geq x/k}
\pr\!\big\{E'(x,k,d,s,A)\big\}.\label{upperSn}\]
Supposing that $w$ is the $j$th node in preorder, $E'(x,k,d,s,A)$ is the event that $\xi_j=d$,
$(\xi_{j+1},\ldots,\xi_{j+s})$ forms a tree, and $(\xi_{j+rs+1,\ldots,j+rs+s}) = (\xi_{j+1},\ldots,
\xi_{j+s})$ for all $1\leq r < k$. Let us say that an integer $j$ is ``good'' if these conditions hold when addition on the indices is done modulo $n$.
Clearly, there are more good $j$ than $j$ satisfying the above conditions. Let $G$ be the event that an index
$j$ chosen uniformly at random from $\{1,\ldots,n\}$ is ``good'';
let $B$ be the event that $(\xi_2,\ldots,\xi_s)$ forms a tree and
$(\xi_{rs+2},\ldots,\xi_{(r+1)s+1}) = (\xi_{j+1},\ldots, \xi_{j+s})$ for all $1\leq r <k$. By a rotational argument due to Dwass~\cite{dwass1969},
\begin{align}
\pr\big\{G\given (\xi_1,\ldots,\xi_n)\ \hbox{forms a tree}\big\} &= \pr\!\Big\{G \;\Big| \sum_{i=1}^n \xi_i = n-1\Big\} \cr
&= {\pr\!\big\{\xi_1 = d\comma B\comma \sum_{i=1}^n \xi_i = n-1\big\}\over \pr\!\big\{\sum_{i=1}^n \xi_i=n-1\big\}}\cr
&= {\pr\!\big\{\xi_1 = d\comma B\comma \sum_{i=\lfloor 1 + ks\rfloor + 1}^n\xi_i = (n-1)-d-k(s-1)\big\}\over
\pr\!\big\{\sum_{i=1}^n \xi_i=n-1\big\}},\cr
\end{align}
so letting
\[R={\pr\!\big\{\sum_{i=1}^{n-(1-ks)}\xi_i=\big(n-(1-ks)-1\big)+(k+1-d)\big\}\over\pr\{\sum_{i=1}^n\xi_i = n-1\}},\]
we have
\[\pr\big\{G\given (\xi_1,\ldots,\xi_n)\ \hbox{forms a tree}\big\} = p_d\pr\{B\}R.\]

Letting $\lambda = \gcd\{i : i \geq 1, p_i>0\}$, a lemma of Kolchin~\cite{kolchin1986} states that uniformly in $y$,
\[\pr\!\bigg\{\sum_{i=1}^n\xi_i = n-y\bigg\} =\begin{cases}
\lambda/\big(\sqrt{2\pi n}\sigma^2\big)\exp\big(\!\!-y^2/2n\sigma^2\big)+{o(1)/\sqrt n},
& \hbox{if $n\bmod\lambda=0$}; \\
0,& \hbox{if $n\bmod \lambda \neq 0$.}\cr\end{cases}
\]
As the $o(1)$ term does not depend on $y$, we find that
\[R = \sqrt{{n-1\over n-(1-ks)-1+(k+1-d)}}\exp\bigg(\!\!-{(1-ks+d-k)^2\over 2\big(n-(1-ks+d-k)\big)}
\sigma^2\bigg)+o(1),\]
where the $o(1)$ term depends only on $n$. Assuming that $ks+d\leq n/2$, we have $R\leq \sqrt2+o(1)$. Hence
\[\pr\big\{G\given (\xi_1,\ldots,\xi_n)\ \hbox{forms a tree}\big\} \leq \big(\sqrt 2 + o(1)\big)p_d\pr\{B\}\]
whenever $ks+d\leq n/2$. We now compute a bound on $\pr\{B\}$. We have
\[\pr\{B\given \xi_2,\ldots,\xi_{1+s}\} = (p_{\xi_2}\cdots p_{\xi_{1+s}})^{k-1}\]
and therefore, by independence of the $\xi_i$,
\begin{align}
\pr\{B\}&=\ex\big\{(p_{\xi_2}\cdots p_{\xi_{1+s}})^{k-1}\one_{[(\xi_2,\ldots,\xi_{1+s})\ \text{\small{forms a tree}}]}
\!\big\}\cr
&\leq \prod_{i=2}^{1+s} \ex\big\{(p_{\xi_i})^{k-1}\big\} = \Big(\sum_{i\geq 0}{p_i}^k\Big)^s.\cr
\end{align}

We can now combine all of these bounds. Substituting everything into \eqref{upperSn}, we have
\begin{align}
\pr\{S_n\geq x\} &\leq n\sum_{k\geq 2}\sum_{d\geq k}{d\choose k}\sum_{s\geq x/k} \big(\sqrt2+o(1)\big)p_d
\Big(\sum_{i\geq 0}{p_i}^k\Big)^s\cr
&\leq\big(\sqrt2+o(1)\big)n\sum_{k\geq 2}\sum_{d\geq k} p_d{d\choose k} \Big(\sum_{i\geq 0}{p_i}^k\Big)^{x/k}
{1\over {1-\sum_{i\geq 0} {p_i}^k}}. \cr
\end{align}
Since the inequality~\eqref{strictineq} was strict, there exists $0<\theta<1$ such that
\begin{align}
\pr\{S_n\geq x\} &\leq {\sqrt2+o(1)\over 1-\sum_{i\geq 0}{p_i}^2} \bigg( n\sum_{d\geq 2} p_d{d\choose 2}
  \Big(\sum_{i\geq 0}{p_i}^2\Big)^{x/2} + n\sum_{k\geq 3}\sum_{d\geq k} p_d{d\choose k}
  \Big(\sum_{i\geq 0}{p_i}^2\Big)^{x/2}\theta^x\bigg)\cr
&\leq {\sqrt2+o(1)\over 1-\sum_{i\geq 0}{p_1}^2} n (\sigma^2 + 1)\Big(\sum_{i\geq 0}{p_i}^2\Big)^{x/2}
  + n\Big(\sum_{i\geq 0}{p_i}^2\Big)^{x/2}\theta^x\sum_{k\geq 3}\sum_{d\leq k}p_d{d\choose k}.\cr
\end{align}
Since
\[\sum_{d\geq 2}p_d\sum_{k=3}^d{d\choose k} \leq \sum_{d\geq 3}p_d2^d\leq \ex\{2^\xi\},\]
we have
\[\pr\{S_n\geq x\}\leq n{\sqrt2(\sigma^2+1)\over 1-\sum_{i\geq 0}{p_i}^2}\Big(\sum_{i\geq 0}{p_i}^2\Big)^{x/2} \big(1+o(1)\big),
\]
provided that $\ex\{2^\xi\}<\infty$. Setting
\[x = (1+\eps){2\log_2n\over \log_2 \big(1/\sum_{i\geq 0} {p_i}^2\big)}= (1+\eps){2\log_2 n\over \Eta_2},\]
we find that $\pr\{S_n\geq x\}\to 0$ as $n\to\infty$.
\end{proof}
\goodbreak

The next lemma presents a lower bound for $S_n$.

\begin{lem}
\label{lowerbound}
Let $\xi$ be an offspring distribution with mean 1 and nonzero finite variance $\sigma^2$.
If $S_n$ is the multiplicity of a conditioned Bienaym\'e--Galton--Watson tree of size $n$ with offspring
distribution $\xi$, then
\[\pr\!\bigg\{S_n<(1-\eps){\log_2 n \over\log_2(1/\gamma)}\bigg\} \to 0 \]
for all $0<\eps<1$, where
$\gamma = \max_{k\geq 2} {p_0}^k {p_k}^{k/(k-1)}$.
\end{lem}

\begin{proof}
Consider a complete $k$-ary tree of height $L$. This tree has $k^L$ leaves and $1+k+\cdots+k^{L-1} = (k^L-1)/(k-1)$ internal nodes, all of degree $k$. The probability that an unconditioned Bienaym\'e--Galton--Watson tree takes this shape is
\[{p_0}^{k^L}{p_k}^{(k^L-1)/(k-1)};\]
call this probability $q$.
For any real number $x$, the statement $S_n<x$ implies that no node in the tree can have the given $k$-ary tree as a subtree for any $k^L\geq x$, as the multiplicity of the $k$-ary tree is $k^L$. Fix $k\geq 2$ for now, let $L$ be the first integer for which $k\geq x$, and let $y = k^L$. Observe that $y\leq kx$. Denote the size of the $k$-ary tree by $z = y+(y-1)/(k-1)$.

\begin{figure}
\begin{center}
  \subfigure{\includegraphics{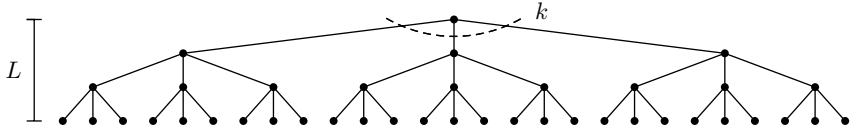}}
  \caption{The construction in the proof of Lemma~\ref{lowerbound}. In this example, both $k$ and $L$ are equal to $3$.}
  \label{fig:embedded}
\end{center}
\end{figure}

We now consider the indices $1, 1+z, 1+2z, 1+3z,\ldots$ in
$\{1,\ldots,n-z\}$. Let $Y_i$ be the event (and ${Y_i}^c$ its complement) that  $(\xi_i,\ldots, \xi_{i+z-1})$ defines precisely the $k$-ary tree, where $i$ is in the set of indices defined above, which has size $\lfloor(n-z)/z\rfloor$. Note that
\[\pr\{S_n<x\}\leq \pr\{S_n< y\} = \pr\bigg\{\bigcap_{i=1}^{n-z} {Y_i}^c\,\Big|\, (\xi_1,\ldots,\xi_n)\ \hbox{defines a tree}\bigg\}.\]
By Dwass' cycle lemma~\cite{dwass1969}, the probability that $(\xi_1,\ldots,\xi_n)$ defines a tree is $\Theta(n^{3/2})$, so
\begin{align}
\pr\{S_n < x\} &\leq \Theta(n^{3/2}) \pr\bigg\{\bigcap_{i=1}^{n-z} {Y_i}^c\bigg\} \cr
&= \Theta(n^{3/2}) \pr\{{Y_i}^c \}^{\lfloor(n-z)/z\rfloor} \cr
&= \Theta(n^{3/2}) (1-q)^{\lfloor(n-z)/z\rfloor} \cr
&\leq \Theta(n^{3/2}) \exp\!\bigg(\!\!-\! \bigg\lfloor {n-z\over z}\bigg\rfloor {p_0}^y {p_k}^{(y-1)/(k-1)}\bigg) \cr
&\leq \Theta(n^{3/2}) \exp\!\bigg(\!\!-\! \bigg\lfloor {n-z\over z}\bigg\rfloor {p_0}^{kx} {p_k}^{(kx-1)/(k-1)}\bigg) \cr
&\leq \Theta(n^{3/2}) \exp\!\bigg(\!\!-\!  \Omega(1)\bigg\lfloor {n-z\over z}\bigg\rfloor \big({p_0}^{k} {p_k}^{(k-1)/(k-1)}\big)^x\bigg) \cr
&\leq \Theta(n^{3/2}) \exp\!\bigg(\!\!-\!  \Omega(1)\bigg\lfloor {n-z\over z}\bigg\rfloor \gamma^x\bigg).\cr
\end{align}
Substituting $(1-\eps)\log_2 n/\log_2(1/\gamma)$ for $x$, and noting that $z = \Theta(\log n)$, we observe that this bound tends to $0$.
\end{proof}

\section{The maximal leaf-degree}
\label{sec:maxleafdeg}

Let $T_n$ be a random critical Bienaym\'e--Galton--Watson tree of size $n$. We let $\xi_u$ be the degree of the node $u$ and let $\lambda_u$ be the number of children of $u$ that are leaves in $T_n$, i.e., the {\it leaf-degree} of $u$. We denote by $L_n$ the random variable $\max_{u\in T_n} \lambda_u$; it is clear that the multiplicity $S_n$ satisfies
$M_n\geq L_n$. The next lemma shows that when the tail of the offspring distribution $\xi$ decays at a rate slower than exponential,
the ratio $L_n/\log n\to \infty$ in probability.
So while our condition in the upper bound that $\ex\{2^\xi\}$ be finite might have seemed somewhat artificial at first glance, we essentially cannot do without it.

\begin{lem}
\label{maxleafdeg}
Let $\ex\{\xi\}=1$, $\var\{\xi\} = \sigma^2\in (0,\infty)$, and suppose that $\ex\{\rho^\xi\}=\infty$ for every $1 < \rho<\infty$. Let $L_n$ be the maximal leaf-degree in $T_n$, the Bienaym\'e--Galton--Watson tree induced by $\xi$, of size $n$. Then
\[{L_n\over \log n}\to \infty\]
in probability along a subsequence, as $n\to \infty$.
\end{lem}

The proof of this lemma uses Kesten's limit tree~\cite{kesten1986} for the offspring distribution $\xi$, whose construction we briefly recall here (see also~\cite{lyonsperes}).
 Kesten's infinite tree, denoted $T_\infty$, is obtained by iterating the following step. Let the root of $T_\infty$ be marked. A marked node has $\zeta$ children, where $\pr\{\zeta=i\} = ip_i$ and $p_i = \pr\{\xi=i\}$. Observe that $\zeta \geq 1$ and $\ex\{\zeta\} = \ex\{\xi^2\} = \sigma^2+1$. Of these $\zeta$ children, a random child is marked and all others are unmarked. The unmarked nodes are roots of independent (unconditioned) Bienaym\'e--Galton--Watson trees. The procedure is then repeated for the sole marked node.

\begin{proof}
We argue by coupling $T_n$ with $T_\infty$. Let $(T_n,k)$ and $(T_\infty,k)$ denote the truncations of $T_n$ and $T_\infty$, respectively, to the nodes at distance $\leq k$ from the root. Then, denoting the total variation distance by $\TV$, it is known that
\[\TV\big((T_n,k_n), (T_\infty, k_n)\big) = o(1)\]
if the sequence $(k_n)$ is $o(\sqrt n)$ (see, e.g.,~\cite{kersting1998} and~\cite{stufler2019}). We couple $(T_n, k_n)$ and $(T_\infty, k_n)$
such that
\[\pr\!\big\{(T_n,k_n) \neq (T_\infty,k_n)\big\} =\TV\big((T_n,k_n), (T_\infty,k_n)\big)\to 0.\]
To show that $L_n/\log n\to \infty$ in probability, it suffices to show
this for $L_n'$, the maximal leaf-degree among all marked nodes of $(T_\infty,k_n)$ at distance $<k_n$ from the root. Let $\zeta_0,\zeta_1,\ldots,\zeta_{k_n-1}$ be the degrees of the marked nodes in $(T_\infty, k_n)$, indexed by their distance from the root, let $\lambda_i$ be the leaf-degree corresponding to $\zeta_i$. Now, fix a constant $c$ and let $A_i$ be the event that $\lambda_i\leq c\log n$; we have
\begin{align}
\pr\{L_n'\leq c\log n\} &\leq \pr\!\bigg\{ \bigcap_{i=0}^{k_n-1} A_i\bigg\} \cr
&= \pr\{A_0\}^{k_n-1} \cr
&\leq \exp\!\big(\!-\!(k_n-1) \pr\{\lambda_0 > c\log n\}\big). \cr
\end{align}
Setting $k_n = \lceil n^{1/3}\rceil + 1$, we have
\[\pr\{L_n' \leq c\log n\}
\leq \exp\!\big(\!-n^{1/3} \pr\{\lambda_0 > c\log n\}\big).\]
Note that $\lambda_0\sim \Bin(\zeta_0 - 1,p_0)$, so that $\pr\{\lambda_0 \leq p_0\zeta_0/2 \given \zeta_0\}\leq 1/2$ for $\zeta_0$ large enough, by the law of large numbers. Therefore, for $n$ large enough, we have
\[\pr\{\lambda_0 > c\log n\}\geq \pr\!\big\{\lambda_0 \geq {p_0\zeta_0\over 2} > c\log n\big\} \geq {1\over 2} \pr\!\bigg\{\zeta_0 > {2c\over p_0}\log n\bigg\}.\]
To conclude the proof, we must show that $n^{1/3}\pr\{\zeta_0 > 2c\log n/p_0\}\to \infty$ along a subsequence of $n$. Note that if $\ex\{\rho^\xi\} = \infty$, then $\int_0^\infty \pr\{\rho^\xi > x\}\d x = \infty$, and thus
\[\sum_{\ell=1}^\infty 2^\ell \pr\!\bigg\{\xi>{\ell\over \log_2\rho}\bigg\} \geq \sum_{\ell=1}^\infty 2^\ell \pr\{\rho^\xi>2^\ell\} \geq \sum_{\ell=1}^\infty \int_{2^\ell}^{2^{\ell+1}} \pr\{\rho^\xi > x\}\d x = \infty,\]
and consequently, $\pr\{\xi > \ell/\log_2\rho\} \geq \ell^{-2}2^{-\ell}$ for infinitely many $\ell \in \N$. As
\[\pr\!\bigg\{\zeta > {\ell\over\log_2\rho}\bigg\} \geq {\ell\over \log_2\rho} \pr\!\bigg\{\xi > {\ell\over\log_2\rho}\bigg\},\]
we see that
\[\pr\!\bigg\{\zeta > {\ell \over \log_2\rho}\bigg\} \geq {1\over \log_2\rho\cdot\ell 2^\ell}\]
for infinitely many $\ell$. Setting $\ell = (2c/p_0)\log n \log_2 \rho$, we have,
\[n^{1/3} \pr\!\bigg\{\zeta > {2c\over p_0}\log n\bigg\}\geq n^{1/3}\cdot
{1\over 2^{2c\log n\log_2\rho/p_0}}\cdot {1\over \log_2\rho \cdot 2c\log n\log_2\rho/p_0}\]
for infinitely many $n$ provided that
\[{2c\over p_0} \log 2 \log_2\rho \leq {1\over 6},\]
which is possible by making $\rho > 1$ small enough. Thus, for every
$c>0$,
\[\limsup_{n\to \infty} \pr\{L_n' > c\log n\} = 1,\]
which is what we wanted to show.
\end{proof}

Note that if for every $\rho < 1$, $p_n > \rho^n$ for all $n$ large enough, then $L_n/\log n\to \infty$ in probability (instead of just along a subsequence).

\section{Examples}
\label{sec:examples}

There exists an important link between certain offspring distributions of conditioned Bienaym\'e--Galton--Watson trees and families of ``simply-generated trees''~\cite{meirmoon1978}. In this section we examine several important families of trees in the Bienaym\'e--Galton--Watson context, and give explicit asymptotic upper and lower bounds for the multiplicity. In each case,
the two important parameters will be
\[\gamma = \max_{k \geq 2}{p_0}^k {p_k}^{k/(k-1)}\qquad\hbox{and}\qquad\Eta_2 = \log_2{1\over \sum_{i\geq 0} {p_i}^2}.\]
We must also verify that $\ex\{2^\xi\}$ is finite,  if the upper bound is to hold. In particular, this latter condition always holds if $\xi$ is bounded. A summary of this section is displayed in Table 1.

\subsection{Full binary trees}

These are trees in which every node must have exactly zero or two children, and arise from the distribution $p_0 = p_2 = 1/2$. We compute $\gamma=1/16$ and $\Eta_2 = 1$, so that
\[(1-\eps){\log_2 n\over 4} \leq S_n \leq (1+\eps)2\log_2 n\]
asymptotically in probability. Because the multiplicity in a full binary tree must be a power of $2$, in essence this means that there exists a sequence of integers $(a_n)$ such that
\[\pr\!\big\{S_n \in \{2^{a_n}, 2^{a_n+1}, 2^{a_n+2}, 2^{a_n+3}\}\big\}\to 1.\]
In other words, in general one cannot improve the ratio between the upper and lower bounds in
Theorem~\ref{mainthm} to a factor of less than $8+\eps$.

\subsection{Flajolet t-ary trees}
Full binary trees are a special case of a
Flajolet $t$-ary tree for $t=2$
(see~\cite{flajoletsedgewick}, p.~68). In general, these are trees whose non-leaf nodes each have $t$ children, and they arise from the finite distribution $p_0 = (t-1)/t$, $p_t = 1/t$. We have
\begin{align}
\gamma &= {p_0}^t{p_t}^{t/(t-1)} \cr
&=\bigg(1-{1\over t}\bigg)^{t} \bigg({1\over t}\bigg)^{t/(t-1)} \cr
&=\exp\!\big(\!\!-1 + o_t(1) - \log t\big),
\end{align}
so $\log_2(1/\gamma) = \log_2 e + \log_2 t + o(1)$ as $t\to \infty$. On the other hand,
\[\Eta_2 = \log_2 {1\over {p_0}^2 + {p_t}^2} = \log_2{1\over
1-2/t + 2/t^2},\]
so $\Eta_2 \sim 2\log_2 e/t$ as $t\to \infty$. This means that as $t$ gets large, the ratio between the upper and lower bound grows as $t \log t$.
\begin{figure}
\begin{equation*}\vcenter{\hspace{-20.86215pt}\vbox{
\footnotesize
\tabskip=.2em plus2em minus.6em
    \halign{
        \hfil$\displaystyle{#}$\hfil & \hfil$\displaystyle{#}$\hfil & \hfil$\displaystyle{#}$\hfil &
        \hfil$\displaystyle{#}$\hfil & \hfil$\displaystyle{#}$\hfil \cr
        \noalign{\hrule}
        \noalign{\medskip}
        \hbox{Family} & \gamma & \Eta_2 & \hbox{Lower bound} & \hbox{Upper bound} \cr
        \noalign{\medskip}
        \noalign{\hrule}
        \noalign{\medskip}
        {\hbox{Full binary} \atop (\Uni\{0,2\})}  &{1\over16} & 1 & {\log_2 n \over 4} & 2\log_2 n \cr
        \noalign{\medskip}
        {\hbox{Flajolet $t$-ary}\atop (p_0 = 1-1/t;\;p_t= 1/t)}  & e^{-1 + o_{t\to\infty}(1) - \log t} & \log_2{1\over 1-2/t-t/t^2} & {\log_2 n\over \log_2 e + \log_2 t + o_{t\to\infty}(1)} & \sim_{t\to\infty} t\log  n \cr
        \noalign{\medskip}
        {\hbox {Cayley} \atop (\Pos(1))} & {1\over 4e^4} & \log_2\Big({e^2 \over I_0(2)}\Big) & {\log_2 n\over 2+4\log_2 e} & {2\log_2 n\over \log_2(e^2 / (I_0(2))} \cr
        \noalign{\medskip}
        {\hbox{Catalan} \atop (\Bin(2,1/2))}   & {1\over 256} & \log_2(8/3) & \log_{256} n & {2\log_2 n\over \log_2(8/3)}\cr
        \noalign{\medskip}
        {\hbox{Binomial} \atop (\Bin(d,1/d))}  & {1 \over 4} {\Big(1-{1 \over d}\Big)}^{4d-2} & \log_2\Big({e^2 \over I_0(2)}\Big)+o_{d\to\infty}(1) & {\log_2 n \over 2 - \log_2({(1-1/d)}^{4d-2})} & {2 \log_2 n \over \log_2(e^2 / (I_0(2))+o_{d\to\infty}(1)} \cr
        \noalign{\medskip}
        {\hbox{Motzkin} \atop (\Uni\{0,1,2\})} &  {1\over 81} & \log_2 3 & \log_{81} n & 2\log_3 n \cr
        \noalign{\medskip}
        {\hbox{Planted plane} \atop (\Geo(1/2))} &  {1\over 256} & \hbox{---} & \log_{256} n & \hbox{---} \cr
        \noalign{\medskip}
        \noalign{\hrule}
    }
}}\end{equation*}
\begin{equation*}\vcenter{\vbox{
\centerline{{\bf Table 1.}\enspace
Leaf multiplicities of certain families of trees}
}}\end{equation*}
\end{figure}

\subsection{Cayley trees}

These trees arise from a $\Pos(1)$ distribution, where $p_i = 1/(e\cdot i!)$ for $i\geq 0$. We verify first that
\[\ex\{2^\xi\} = \sum_{i=0}^\infty {2^i\over e i!} = e < \infty,\]
and then work out that $\gamma = 1/(4e^4)$. Letting
\[I_0(z) = \sum_{i=0}^\infty {(i^2/4)^k\over i!\cdot\Gamma(z+1)} = {1\over \pi}\int_0^\pi e^{z\cos \theta} \d\theta\]
be the modified Bessel function of the first kind (see~\cite{abrasteg1972}, p.~376), we find that
\[\sum_{i=0}^\infty {p_i}^2 = {1\over e^2} \sum_{i=0}^\infty {1\over (i!)^2} = {1\over e^2}I_0(2),\]
meaning that $\Eta_2 = 2\log_2 e - \log_2\!\big(I_0(2)\big)$. Putting everything together, the lower and upper bounds in probability for $S_n$ are, respectively,
\[{\log_2 n\over 2+4\log_2 e} \approx {\log_2 n\over 7.771}
\qquad\hbox{and}\qquad {2\log_2 n\over 2\log_2 e - \log_2\big((1/\pi)\int_0^\pi e^{2\cos \theta}\d \theta\big)} \approx {\log_2 n\over 0.8483}.\]

\subsection{Catalan trees}

When we set $p_0 = p_2 = 1/4$ and $p_1 = 1/2$, we obtain a family of trees often called Catalan trees, since the number of such trees on $n$ nodes is ${2n\choose n}/(n+1)$. There is a one-to-one correspondence between Catalan trees on $n$ nodes and full binary trees on $2n+1$ nodes, since one obtains a full binary tree from a Catalan tree by adding artificial external nodes to every empty slot, and this procedure is reversed by removing all leaves from a full binary tree. It is easy to see that the leaf multiplicity of a full binary tree is exactly double the multiplicity of its corresponding Catalan tree. By plugging in $d=2$ above, the lower bound given by
Lemma~\ref{lowerbound} is $\log_2 n / 8$, which makes sense since the correspondence with full binary trees tells us that the lower bound on the Catalan trees should be similar to $\log_2(2n+1)/8$. We calculate
$\Eta_2 = \log_2(8/3)$ and the upper bound
is $2\log_2 n/\log_2(8/3)$, so
the ratio between the upper and lower bounds is $16/\log_2(8/3)$.

\subsection{Binomial trees}

Catalan trees are a special case of a binomial tree. For integer parameter $d\geq 2$, nodes in these trees have $d$ ``slots'' that may or may not contain a child; so there are ${d\choose i}$ ways for a node to have $i$ children, for $0\leq i\leq d$. These trees correspond to a $\Bin(d, 1/d)$ distribution. We compute
\[\gamma = (p_0 p_2)^2 = \Bigg(\bigg({d-1 \over d}\bigg)^{d}\cdot {d(d-1) \over 2}\cdot {(d-1)^{d-2} \over d^{d}}\Bigg)^2 = {1 \over 4} \bigg(1-{1 \over d}\bigg)^{4d-2}.\]
Note that taking the limit $d \to \infty$, the $\Bin(d, 1/d)$ distributions approach a $\Pos(1)$ distribution. Thus we see from our earlier discussion on the Cayley trees that $\Eta_2 = \log_2\big(e^2 / (I_0(2)\big)+o_{d\to\infty}(1)$. This gives the respective lower and upper bounds
\[{\log_2 n \over 2 - (4d-2)\log_2(1-1/d)} \qquad\hbox{ and }\qquad {2 \log_2 n \over \log_2\big(e^2 / (I_0(2)\big)+o_{d\to\infty}(1)}.\]
The lower bound tends to $\log_2n / (2 + 4\log_2 e)$ as $d \to \infty$, matching the lower bound we obtained for Cayley trees above.

\subsection{Motzkin trees}

Also known as unary-binary trees, these are trees in which each non-leaf node can have either one or two children. They correspond to the distribution with $p_0 = p_1 = p_2 = 1/3$. We easily compute $\gamma = 1/81$ and $\Eta_2 = \log_2 3$, which yields an asymptotic lower bound of $\log_2 n/(\log_2 81)= \log_{81} n$ and an asymptotic upper bound of $2\log_2 n /\log_2 3 = 2 \log_3 n$. The ratio between the upper and lower bounds is $8$.

\subsection{Planted plane trees}

These are trees with ordered children, so that each can be embedded in the plane in a unique way. They correspond to a $\Geo(1/2)$ distribution, with $p_i = 1/2^{i+1}$ for all $i$. We find that $\gamma = 1/256$, so we have the asymptotic lower bound $S_n \geq \log_2 n/ 8$. 
Unfortunately, we have $\ex\{2^\xi\} = \sum_{i\geq 0} 1/2 = \infty$, so Lemma~\ref{upperbound} cannot be applied to give an upper bound here. We note that the maximal degree $\Delta_n$ of $T_n$ satisfies $\Delta_n/\log_2 n\to 1$ in probability (see, e.g., \cite{rootest2020}, Lemma 6). However, this does not imply that $S_n = O(\log n)$ in probability.

\section{Automorphic multiplicity}

\label{sec:automult}

The multiplicity of a tree does not have a natural extension to unrooted trees,
because whether or not two nodes are identical depends crucially on their position in relation to a
distinguished root node $u$. In this section we briefly investigate an alternate notion of multiplicity
that does extend nicely to free trees. It arises in the problem of root estimation
in Galton--Watson trees described in~\cite{rootest2020}. We briefly recall some definitions.
Let $T$ be a rooted tree. By disregarding the parent-child directions of the edges, we obtain a free tree
$T_\F$. Conversely, if we start with a free tree $T_\F$ and any node $u$, we can define a
{\it rooting of $T_\F$ at $u$}
to be the rooted tree $T_u$ obtained by fixing $u$ as the root.
This does not give rise to a unique tree in general, because children of a given node may
hang on the wall in an arbitrary left-to-right order, but our new notion of multiplicity will treat all
of these possible ordered trees the same.

Let $\Aut(T_\F)$ be the group of all graph automorphisms of $T_\F$, that is, bijections $f$ from the
set of vertices $T_\F$
to itself such that for vertices $u$ and $v$, $f(u)$ is adjacent to $f(v)$ whenever $u$ is adjacent to $v$.
We can then define an {\it automorphism} of $T_u$ to be a graph automorphism of $T_\F$ such that
the root $u$ stays fixed. By a slight abuse of notation,
we denote the set of these rooted-tree automorphisms by $\Aut(T_u)$;
formally this is the {\it stabilizer subgroup}
\[\Stab(u) = \{ g\in \Aut(T_\F) : g\cdot u = u\}\]
of $\Aut(T_\F)$.
We will say that two nodes $v$ and $w$ in $T_u$ are {\it congruent} and write $v\sim_u w$
if $v$ and $w$ belong to the same orbit under the action of $\Aut(T_u)$. This means that there exists
an element $f$ of $\Aut(T_u)$ such that $f(v) = w$. It is clear that this gives us an equivalence
relation on the set of all nodes of $T_u$,
and the {\it automorphic multiplicity} of a node $v$, denoted $\mu(u,v)$,
is the size of the equivalence class of $v$ under this relation. Since any node can be mapped to
itself under an automorphism, $\mu(u,v)\geq 1$ for all $v$.

In fact, one can define the relation 
$\sim_u$, and consequently the function $\mu$,
purely in terms of the relation $\equiv$.
We have $v\sim_u w$ if and only if there exists a permutation
for every node in $T_u$ such that applying each permutation to the left-to-right ordering of its
respective node's children results in a tree in which $v\equiv w$. The analogue of $S$ in this setting
is the {\it automorphic (leaf) multiplicity} $M(T)$ of a rooted tree $T$. If $o$ is the root of the tree $T$,
then $M(T)$ is the maximum value of $\mu(o,v)$ over all
nodes $v$ in the tree $T$.

\begin{figure}
\begin{center}
  \subfigure{\includegraphics{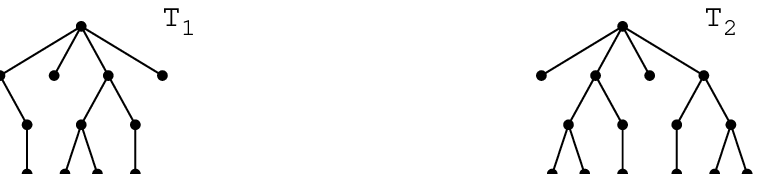}}
  \caption{Different leaf multiplicities but the same automorphic leaf multiplicity.}
  \label{fig:autodef}
\end{center}
\end{figure}

Fig.~\ref{fig:autodef} illustrates the distinction between the automorphic and non-auto\-mor\-phic multiplicity.
We have $S(T_1) = M(T_1) = 4$, since the two non-leaf children of the root have identical
(and therefore congruent) subtrees. In $T_2$, on the other hand, these subtrees are congruent but not
identical, so that $M(T_2) = 4$ but the non-automorphic multiplicity of $T_2$ is only $2$.

This definition is still somewhat at odds with the notion of multiplicity that arises in the
root estimation problem from~\cite{rootest2020}.
In that setting, one considers all graph
automorphisms of the free tree, not just ones that fix the root. We will call the size of the orbit of a node
under this larger action the {\it free multiplicity} $\mu_\F$ and if two nodes $u$ and $v$ are congruent under
an arbitrary graph automorphism, then we write $u\sim_\F v$ and say that the two nodes are {\it free-congruent}.
We also let $M_\F(T)$ denote the {\it free (leaf) multiplicity}, the maximum value of $\mu_\F$
over all nodes in the free tree $T_\F$.

\begin{figure}
\begin{center}
  \subfigure{\includegraphics{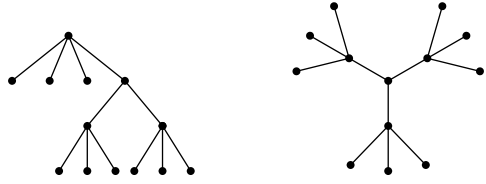}}
  \caption{A rooted tree $T$ with $M(T)=6$ and $M_\F(T)=9$.}
  \label{fig:freecong}
\end{center}
\end{figure}

Fig.~\ref{fig:freecong} shows the relation between the automorphic multiplicity of a rooted tree and the
free multiplicity its free-tree counterpart. Note that $M(T) \leq M_\F(T)$ for any rooted tree $T$,
since we have $\mu(u) \leq \mu_\F(u)$ for every node $u$.
We shall spend the rest of this section showing that this inequality can more or less be reversed. First,
we need three lemmas, and in their statements and proofs, we shall understand ``multiplicity'' to mean
``free multiplicity''. In the following proof, we also write $[G:H]$ to denote the {\it index} of a subgroup $H$
in a larger group $G$; that is, the cardinality of the coset space $G/H$.

\begin{lem}
\label{sakslemma}
If $u$ and $v$ are adjacent nodes in a finite free tree $T$, then
either $\mu_\F(u)$ is an integer multiple of $\mu_\F(v)$ or the other way around.
\end{lem}

\begin{proof}
We may reduce to the case where one of $u$ or $v$ is a leaf. This is because if neither is a leaf,
then it is not in the orbit of any leaf by graph automorphism, so we can remove all the leaves from
the tree $T$ without changing either of $\mu_\F(u)$ or $\mu_\F(v)$. This is done finitely many times
since $T$ is finite and always contains at least one leaf.

Now without loss of generality, suppose $u$ is the leaf and $v$ is its unique neighbour. By the
orbit-stabilizer theorem,
\[\bigl|\Aut(T)\bigr| = \mu_\F(u) \bigl|\Stab(u)\bigr| = \mu_\F(v)\bigl|\Stab(v)\bigr|,\]
where stabilizers are taken with respect to the group $\Aut(T)$. Every automorphism fixing $u$ must
permute its neighbours, but since $u$ only has one neighbour, we have $\Stab(u)\subseteq \Stab(v)$.
Thus
\begin{align}
\mu_\F(u)
&= {\mu_\F(v) \bigl|\Stab(v)\bigr|\over \bigl|\Stab(u)\bigr|} \cr
&= {\mu_\F(v) \bigl[\Stab(v):\Stab(u)\bigr] \bigl|\Stab(u)\bigr|\over \bigl|\Stab(u)\bigr|} \cr
&= \mu_\F(v) \bigl[\Stab(v):\Stab(u)\bigr],\cr
\end{align}
proving the lemma.
\end{proof}

The next lemma formalizes the intuitive
notation that in a free tree, the multiplicities are in some sense smaller towards the centre of the tree.

\begin{lem}
\label{middlemult}
Let $u \edge v \edge w$ be neighbouring nodes in a free tree $T$ with
$v$ being the central node. Then $v$ cannot have strict maximal free multiplicity among the three nodes;
that is, $\mu_\F(v) \leq \mu_\F(u)$ or $\mu_\F(v) \leq \mu_\F(w)$.
\end{lem}

\begin{proof}
Suppose for contradiction that $\mu_\F(v) > \mu_\F(u)$ and $\mu_\F(v) > \mu_\F(w)$. Then, for each of the
pairs of neighbours
$u\edge v$ and $v\edge w$,
the multiplicity of one of the nodes must be an integer multiple of the multiplicity of the other,
by the previous lemma.
So there must be integers $r,s > 1$ such that
\[\mu_\F(v) = r \mu_\F(w)\quad\hbox{and}\quad \mu_\F(v) = s \mu_\F(u). \]
The situation is illustrated in Fig.~\ref{fig:landscape}. Since $\mu_\F(v) = s \mu_\F(u)$, $u$ must
have $s - 1$ children in the orbit of $v$ and thus have subtree rooted at each of these
children be isomorphic to $B$. Similarly, since $\mu_\F(v) = r \mu_\F(w)$, $w$ must have $r - 1$
child subtrees isomorphic to $A$.

\begin{figure}
\begin{center}
  \subfigure{\includegraphics{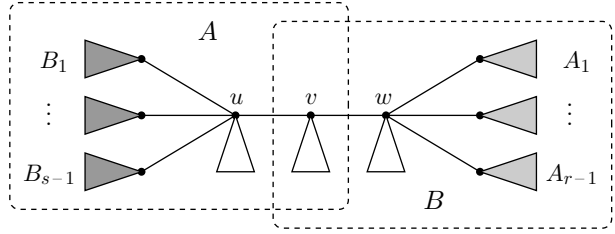}}
  \caption{Three adjacent nodes and their subtrees.}
  \label{fig:landscape}
\end{center}
\end{figure}

We note that in order to satisfy the $r,s>1$ requirements, we must have
\[|A| \geq (s - 1) |B| + 2 \quad\hbox{and}\quad |B| \geq (r - 1)|A| + 2,\]
where the additional $+2$ terms come respectively from nodes $u$ and $v$ (for $|A|$) or $v$ and
$w$ (for $|B|$). This implies that
\[|A| \geq (s - 1)(r - 1)|A| + 2s,\]
which is impossible if $|A|\geq 1$ and $r,s > 1$. The contradiction tells us that $v$
cannot have strict maximal multiplicity among the three nodes.
\end{proof}

We have established that if we embed a free tree into the $(x,y)$-plane and then lift the nodes
up by setting each node's $z$-coordinate to its multiplicity, then the result is a convex,
spidery bowl or valley. This is illustrated in Fig.~\ref{fig:spidery}.

\begin{figure}
\begin{center}
  \subfigure{\includegraphics{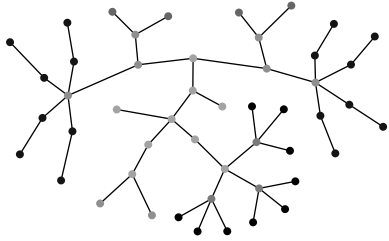}}
  \caption{Darker shades of grey indicate higher multiplicities in this free tree.}
  \label{fig:spidery}
\end{center}
\end{figure}

On a path between any two endpoints, the multiplicities decrease monotonically towards the centre
of the tree before increasing monotonically towards the endpoint. There is a central connected core
of nodes of minimal multiplicity and we are able to show
that this minimal multiplicity cannot be greater than 2.

\begin{lem}
\label{minmult}
If $F = (V,E)$ is a finite free tree, then the node of minimal multiplicity
in $F$ has multiplicity 1 or 2.
\end{lem}

\begin{proof}
The proof is by contraposition. Let $u\in V(F)$ be a node of minimal multiplicity and suppose
$\mu_\F(u) >2$. Let $C_u$ be the orbit of $u$. There is a subtree
$F'$ whose endpoints are the members of $C_u$; since $m>2$ and the graph is connected, there is
necessarily at least one node $v\in F'\setminus C_u$. By Lemma~\ref{middlemult},
we have $\mu_\F(v)\leq \mu_\F(u)$ but by
minimality of $\mu_\F(u)$, we know that $\mu_\F(v) = \mu_\F(u)$.
So we can repeat the argument with $C_v$ to find that
the tree is infinite (at each step we are removing $\mu_\F(u)$ nodes from the free tree, but the process never
terminates).

Note that this argument does not work when $\mu_\F(u) = 2$ because $F'$ may simply consist of two nodes
connected by one edge.
\end{proof}

\begin{thm}
Let $T$ be a rooted tree with $n$ nodes; let $M(T)$ and $M_\F(T)$ be the
automorphic multiplicity and
free multiplicity of $T$, respectively. We have the inequality
\[M_\F(T) \leq 2M(T),\]
and this bound is the best possible.
\end{thm}

\begin{proof}
Suppose first that $n\geq 3$.
Let $v$ be a leaf of maximal automorphic multiplicity in $T_\F$, and let $[v]$ denote the set of nodes that
are free-congruent to $v$ (so $\bigl|[v]\bigr| = M_\F(T)$).
By Lemma~\ref{minmult}, a node $s$ of minimal automorphic multiplicity
either has $\mu_\F(s) = 1$ or $\mu_\F(s) = 2$, and since we assumed that $n\geq 3$, we can require that $s$ not be
a leaf.

If $\mu_\F(s) = 1$, then $M(T_s) = M_\F(T)$, since any automorphism of $T_\F$ already fixes $s$. The nodes
in $[v]$ all lie in some subtrees of $s$, and without loss of generality, we may assume that they do not
all lie in the same subtree, since if $s'$ is the only child of $s$ whose subtree contains nodes of $[v]$,
we can reroot the tree $T_s$ at $s'$ instead without changing the maximum automorphic multiplicity.
There are $d\geq 2$ children of the root whose subtrees contain elements of $[v]$; each one
contains an equal proportion of these nodes, so $d$ divides $M_\F(T)$. If we reroot the tree
at any node outside these subtrees, then the automorphic multiplicity of the tree does not change.
If, on the other hand,
we choose a node in one of these subtrees, then there are still $(d-1)M_\F(T)/d$ leaves that can still be shuffled
amongst themselves, so the maximum automorphic multiplicity is $(d-1)M_\F(T)/d \geq M_\F(T)/2$.

If $\mu_\F(s) = 2$, there is a node $s'$ that is free-congruent to $s$, and there is mirror symmetry in the graph.
This means that there is a way to split the graph along an edge such that the two sides
have the exact same shape, one
contains $s$, and the other contains $s'$. The side containing $s$ has $M_\F(T)/2$ members of $[v]$; call this
half $[v]_s$ and the other half $[v]_{s'}$.
When the tree is rooted at $s$, we find that $M(T_s) = M_\F(T)/2$, since any two members of $[v]_s$ can be
exchanged and any two members of $[v]_{s'}$ can be exchanged
(but exchanges cannot happen between the two subtrees).
And rerooting the tree at an arbitrary node, it is clear that the automorphic multiplicity of the
tree will not decrease.

When $n=1$ the statement is trivial, and taking $n=2$ shows that the bound is the best possible,
because if $T$ is the tree with a root and a single (leaf) child, then $M_\F(T) = 2$ and $M(T) = 1$.
\end{proof}

This theorem tells us that the asymptotics of the free multiplicity are the same as the asymptotics
of the automorphic
multiplicity, up to a fudge factor of $2$.

Because congruence of two nodes is immediately implied by their being identical under $\equiv$,
we have $S(T)\leq M(T)$ for all
rooted trees $T$.
Thus if $M_n = M(T_n)$ and $F_n = M_\F(T_n)$, where $T_n$ is a conditioned Galton--Watson tree
of size $n$, then if $\gamma$ is as defined in Lemma~\ref{lowerbound}, the inequality
\[F_n \geq M_n \geq (1-\eps){\log_2 n \over\log_2(1/\gamma)}\]
holds with probability tending to 1.

\acknowledgements
\label{sec:ack}
We thank the anonymous referee for numerous insightful comments that substantially improved the
the readability and rigour of the paper. We also thank Jonah Saks for helping us find a clean proof of
Lemma~\ref{sakslemma}.

\bibliography{multiplicity-latex}         
\bibliographystyle{abbrv}     

\end{document}

%% file: mathmacros.tex
\def\xskip{\hskip 7pt plus 3pt minus 4pt}

\def\slug{\quad\hbox{\kern1.5pt\vrule width2.5pt height6pt depth1.5pt\kern1.5pt}\medskip}
\def\noskipslug{\quad\hbox{\kern1.5pt\vrule width2.5pt height6pt depth1.5pt\kern1.5pt}}

\newdimen\algindent
\newif\ifitempar \itempartrue 
\def\algindentset#1{\setbox0\hbox{{\bf #1.\kern.25em}}\algindent=\wd0\relax}
\def\algbegin #1 #2{\algindentset{#21}\alg #1 #2} 
\def\aalgbegin #1 #2{\algindentset{#211}\alg #1 #2} 
\def\alg#1(#2). {\medbreak 
  \noindent{\bf#1}({\it#2\/}).\xskip\ignorespaces}
\def\algstep#1.{\ifitempar\smallskip\noindent\else\itempartrue
  \hskip-\parindent\fi
  \hbox to\algindent{\bf\hfil #1.\kern.25em}%
  \hangindent=\algindent\hangafter=1\ignorespaces}

\def\N{{\bf N}}  

\def\op#1{\mathop{\hbox{#1}}\nolimits}

\def\pr{\mathop{\hbox{\bf P}}\nolimits}
\def\ex{\mathop{\hbox{\bf E}}\nolimits}
\def\var{\mathop{\hbox{\bf V}}\nolimits}
\def\one{\mathop{\hbox{\bf 1}}\nolimits}

\def\edge{\mathrel-\!\!\mathrel-}

\newcount\thmcount  
\newcount\sectcount  
\newcount\figcount  
\newcount\eqcount  
\def\oldno#1{\eqno({\oldstyle#1})}

\def\adveq{\oldno{\the\eqcount}\global\advance\eqcount by 1}  

\def\advsect{\section\the\sectcount\global\advance\sectcount by 1. }

\outer\def\parenproclaim #1 (#2).#3\par{\medbreak
  \noindent{\bf #1}\enspace\rm({\it #2\/}).\nobreak\ignorespaces{\sl #3\par}
  \ifdim\lastskip<\medskipamount \removelastskip\penalty55\medskip\fi}

\def\Bin{\op{{\rm binomial}}}
\def\Pos{\op{{\rm Poisson}}}
\def\Geo{\op{{\rm geometric}}}
\def\Uni{\op{{\rm uniform}}}
\def\TV{{\rm TV}}

\def\advthm{\the\thmcount\global\advance \thmcount by 1}
\def\eps{\epsilon}
\def\d{\,d}
\def\Eta{{\rm H}}

\def\pmax{p_{\rm max}}
\def\A{{\cal A}}
\def\F{{\rm F}}
\def\given{\,|\,}
\def\comma{,\,}
\def\Stab{{\rm Stab}}
\def\Aut{{\rm Aut}}

%% file: multiplicity-latex.bbl
\begin{thebibliography}{10}

\bibitem{abrasteg1972}
M.~Abramowitz and I.~Stegun, editors.
\newblock {\em Handbook of Mathematical Functions with Formulas, Graphs, and
  Mathematical Tables}.
\newblock U.S. Government Printing Office, Washington, 1972.

\bibitem{athreya1972branching}
K.~Athreya and P.~Ney.
\newblock {\em Branching Processes}.
\newblock Springer Verlag, Berlin, 1972.

\bibitem{bienayme1845}
I.-J. Bienaym\'e.
\newblock De la loi de multiplication et de la dur\'ee des familles.
\newblock {\em {Soc. Philomath. Paris Extraits}}, {\bf 5}:37--39, 1845.

\bibitem{rootest2020}
A.~M. Brandenberger, L.~Devroye, and M.~K. Goh.
\newblock Root estimation in {G}alton-{W}atson trees.
\newblock {\em arXiv preprint 2007.05681}, 2020.

\bibitem{dwass1969}
M.~Dwass.
\newblock The total progeny in a branching process.
\newblock {\em Journal of Applied Probability}, {\bf 6}:682--686, 1969.

\bibitem{flajoletsedgewick}
P.~Flajolet and R.~Sedgewick.
\newblock {\em Analytic Combinatorics}.
\newblock Cambridge University Press, 2009.

\bibitem{galtonwatson1874}
F.~Galton and H.~W. Watson.
\newblock On the probability of extinction of families.
\newblock {\em {J. Anthropol. Inst.}}, {\bf 4}:138--144, 1874.

\bibitem{jacquet1999entropy}
P.~Jacquet and W.~Szpankowski.
\newblock Entropy computations via analytic depoissonization.
\newblock {\em IEEE Transactions on Information theory}, {\bf
  45}(4):1072--1081, 1999.

\bibitem{jagers2009}
P.~Jagers.
\newblock Some notes on the history of branching processes, from my
  perspective, 2009.

\bibitem{kersting1998}
G.~Kersting.
\newblock On the height profile of a conditioned {G}alton--{W}atson tree,
  preprint.
\newblock {\em arXiv:1101.3656}, 1998.

\bibitem{kesten1986}
H.~Kesten.
\newblock Subdiffusive behavior of a random walk on a random cluster.
\newblock {\em Annales de l'{Institut Henri Poincar\'e Probability and
  Statistics}}, {\bf 22}:425--487, 1986.

\bibitem{kolchin1986}
V.~F. Kolchin.
\newblock {\em Random Mappings}.
\newblock Optimisation Software Inc., New York, 1986.

\bibitem{lyonsperes}
R.~Lyons and Y.~Peres.
\newblock {\em Probability on Trees and Networks}.
\newblock Cambridge University Press, New York, 2016.

\bibitem{meirmoon1978}
A.~Meir and J.~Moon.
\newblock On the altitude of nodes in random trees.
\newblock {\em Canadian Journal of Mathematics}, {\bf 30}(5):997--1015, 1978.

\bibitem{renyi1961measures}
A.~R{\'e}nyi.
\newblock On measures of entropy and information.
\newblock In {\em Proceedings of the Fourth Berkeley Symposium on Mathematical
  Statistics and Probability}, pages 547--561, 1961.

\bibitem{shannon1948}
C.~E. Shannon.
\newblock A mathematical theory of communication.
\newblock {\em Bell System Technical Journal}, {\bf 27}(3):379--423, 1948.

\bibitem{stufler2019}
B.~Stufler.
\newblock {Local limits of large Galton--Watson trees rerooted at a random
  vertex}.
\newblock {\em Annales de l'Institut Henri Poincar{\'e}, Probabilit{\'e}s et
  Statistiques}, {\bf 55}(1):155--183, 2019.

\end{thebibliography}
